\documentclass[12pt,a4paper]{amsart}
\usepackage{t1enc}
\usepackage[latin1]{inputenc}
\usepackage{amssymb}
\usepackage{amsthm}
\usepackage{amsmath}
\usepackage{latexsym}
\usepackage{array}
\usepackage{mathrsfs}
\usepackage[all]{xy}
\usepackage{graphics}
\usepackage{xcolor}
\usepackage{stmaryrd}
\usepackage{color}
\usepackage{enumerate}
\title[Rationality of cycles over function field of $\mathbf{SL_1}(A)$-torsors]{Rationality of algebraic cycles over function field of $\mathbf{SL_1}(A)$-torsors}
\date{2 September 2014}
\author{{Raphael Fino}}
\address
{UPMC Sorbonne Universit\'es\\
Institut de Math\'ematiques de Jussieu\\
Paris\\
FRANCE}
\address
{{\it Web page:}
{\tt www.math.jussieu.fr/\~{ }fino}}
\email {raphael.fino {\it at} imj-prg.fr}
\numberwithin{equation}{section}
\theoremstyle{definition}

\newtheorem{rem}[equation]{Remark}

\newtheorem{prop}[equation]{Proposition}
\newtheorem{thm}[equation]{Theorem}

\newcommand{\twoheadlongrightarrow}{\relbar\joinrel\twoheadrightarrow}
\begin{document}

\begin{abstract}
In this note we prove a result comparing rationality of algebraic cycles over the function field
of a $\mathbf{SL_1}(A)$-torsor for a central simple algebra $A$ 
and over the base field.

\end{abstract}

\maketitle
\tableofcontents

\section{Introduction}

Let $A$ be a central simple algebra over a field $F$ and let $\text{Nrd}:A^{\times}\rightarrow F^{\times}$
be the reduced norm homomorphism .
We recall that the homomorphism $F^{\times}\rightarrow H^1(F,\mathbf{SL_1}(A))$,
associating to $c\in F^{\times}$ the $\mathbf{SL_1}(A)$-torsor $X_c$ given
by the equation $\text{Nrd}=c$, is surjective (with kernel $\text{Nrd}(A^{\times})$) -- see \cite[Proposition 2.7.3]{CSAGC} for instance.

\medskip

The main purpose of this note is to prove the following theorem dealing with rationality of algebraic cycles over 
function field of $\mathbf{SL_1}(A)$-torsors.

\begin{thm}
\textit{Let $A$ be a central simple algebra of prime degree $p$ over a field $F$ and let $X$ be a $\mathbf{SL_1}(A)$-torsor.
Then}
\begin{enumerate}[(i)]
\item \textit{for any equidimensional $F$-variety $Y$, the change of field homomorphism}
\[\text{CH}(Y)\rightarrow \text{CH}(Y_{F(X)}),\]
\textit{where} $\text{CH}$ \textit{is the integral Chow group, is surjective in codimension $< p+1$.}

\item \textit{it is also surjective in codimension $p+1$
for a given $Y$ provided that the variety $X_{F(\zeta)}$ does not have any closed point of prime to $p$ degree for each generic point $\zeta \in Y$.}
\end{enumerate}
\end{thm}

The method of proof mainly relies on the following statement. 
This proposition is a version of the result
\cite[Lemma 88.5]{EKM} slightly altered to fit our situation 
(see also the proof of \cite[Proposition 2.8]{OSNV}).

\begin{prop}[Karpenko, Merkurjev]
\textit{Let $X$ be a smooth variety, and $Y$ an equidimensional variety. Given an integer $m$ such that for any nonnegative integer $i$ and any point $y\in Y$ of codimension $i$
the change of field homomorphism }
\[\text{CH}^{m-i}(X)\longrightarrow \text{CH}^{m-i}(X_{F(y)})\]
\textit{is surjective, the change of field homomorphism}
\[\text{CH}^m(Y)\longrightarrow \text{CH}^m(Y_{F(X)})\]
\textit{is also surjective.}
\end{prop}
The proof of Theorem 1.1 is given in Section 3.
In Section 4, we describe how this theorem can be related to a similar result
dealing with rationality of algebraic cycles over function field of projective
homogeneous varieties under some groups of exceptional type.

\medskip

\smallskip
\noindent
{\sc Acknowledgements.}
This note has been conceived while I was visiting the University of Alberta and
I would like to thank the Department of Mathematical and Statistical Sciences for the hospitality.
I also would like to thank Nikita Karpenko for suggestions having considerably improved this note.

\section{Preliminaries}

\subsection{Topological filtration and Chow groups}

For any smooth variety $X$ over a field $F$ (in this paper, an $F$-variety is a separated scheme of finite type over $F$), one can consider the topological filtration on the Grothendieck ring $K_0(X)$, whose term of 
codimension $i$ is given by
\[\tau^i(X)=\langle [\mathcal{O}_Z]\,|\,Z\hookrightarrow X \: \text{and}\: \text{codim}(Z)\geq i \rangle,\]
\noindent where $[\mathcal{O}_Z]$
is the class in $K_0(X)$ of the structure sheaf of a closed subvariety $Z$.
We write $\tau^{i/i+1}(X)$ for the successive quotients.
We denote by $pr^i$ the canonical surjection
\[\begin{array}{rcl}
\text{CH}^i(X) & \twoheadlongrightarrow & \tau^{i/i+1}(X) \\
\left[Z\right] & \longmapsto & [\mathcal{O}_Z]
\end{array},\]
where \text{CH} is the integral Chow group.
By the Riemann-Roch Theorem without denominators the $i$-th Chern class induces an homomorphism in the opposite way
$c_i : \tau^{i/i+1}(X) \rightarrow \text{CH}^i(X)$ such that the composition $c_i \circ pr$ is the multiplication by $(-1)^{i-1}(i-1)!$.

\medskip

Note that for any prime $p$, one can also consider the topological filtration $\tau_p$ on the ring $K_0(X)/pK_0(X)$ by replacing $K_0(X)$ by $K_0(X)/pK_0(X)$ in the previous definition.
In particular, we get that for any $0\leq i \leq p$, the map $pr_p^i: \text{Ch}^i(X)  \twoheadrightarrow  \tau_p^{i/i+1}(X)$, where $\text{Ch}$ is the Chow group modulo $p$, is an isomorphism.

\begin{rem}
Assume that $X$ is a $\mathbf{SL_1}(A)$-torsor and let $p$ be a prime. 
One has $K_0(X)=\mathbb{Z}$ by the result \cite[Theorem A]{Pan2} of I.\,Panin and consequently,
for $i\geq 1$, the term $\tau^{i}(X)$ is equal to zero. Therefore, for any $1\leq i\leq p$, one has $\text{Ch}^{i}(X)=0$.
Moreover, by the result \cite[Theorem 2.7]{sus91} of A.\,Suslin, one has $\text{CH}^i(\mathbf{SL_p})=0$
for any $i\geq 1$.
Hence, for
$A$ of degree $p$ (then there exists a splitting field of $A$ of degree $p$),
it follows by transfert argument that $p\cdot \text{CH}^i(X)=0$ for any $i\geq 1$.
Therefore, for $X$ a $\mathbf{SL_1}(A)$-torsor, with $A$ of prime degree $p$, one has $\text{CH}^{i}(X)=0$ for any $1\leq i\leq p$. 
Note that, by Proposition 1.2, this gives Theorem 1.1(i) already.
\end{rem}

\medskip

\subsection{Brown-Gersten-Quillen spectral sequence}

For any smooth variety $X$ and any $i\geq 1$,
the epimorphism $pr^i$ coincides with the edge homomorphism of the spectral Brown-Gersten-Quillen
structure $E^{i,-i}_2(X)\Rightarrow K_0(X)$ (see \cite[\S 7]{HKT1}),
that is to say 
\[pr^i: \text{CH}^{i}(X)\simeq E^{i,-i}_2(X)\twoheadrightarrow 
 \cdots \twoheadrightarrow
E^{i,-i}_{i+1}(X)=\tau^{i/i+1}(X).\]

Assume that $X$ is a $\mathbf{SL_1}(A)$-torsor, with $A$ of prime degree $p$.
Then it follows from Remark 2.1 that $E^{i,-i}_{i-1}(X)=0$ for $3\leq i\leq p$.
Consequently, one has $A^1(X,K_2)=E_p^{1,-2}(X)$.

\medskip

Moreover, by the result \cite[Theorem 3.4]{Adams} of A.\,Merkurjev, for any smooth variety $X$, every prime divisor $l$
of the order of the differential $\delta_r$ ending in $E^{p+1,-p-1}_{r}(X)$ is such that $l-1$ divides $r-1$. 
Therefore, for any prime $p$ and $2\leq r\leq p-1$, the differential $\delta_r$ is of prime to $p$ order. 
Assume furthermore that $X$ is a $\mathbf{SL_1}(A)$-torsor, with $A$ of prime degree $p$.
Since $p\cdot \text{CH}^{p+1}(X)=0$ (see Remark 2.1), one deduce that, for $2\leq r\leq p-1$, the differential $\delta_r$ is trivial. Consequently, one has $\text{CH}^{p+1}(X)=E_p^{p+1,-p-1}(X)$.

\medskip

Therefore, for $X$ a $\mathbf{SL_1}(A)$-torsor, with $A$ of prime degree $p$, the
differential $\delta_p$ in the BGQ-structure is a homomorphism
\[\delta : A^1(X,K_2)\rightarrow \text{CH}^{p+1}(X).\]

\begin{rem}
Let $X$ be a principal homogeneous space for a semisimple group $G$.
By \cite[Part II, Example 4.3.3 and Corollary 5.4]{Coh}, one has $E_2^{0,-1}(X)=A^0(X, K_1)= F^{\times}$  and the composition
 $F^{\times}=K_1(F)\rightarrow K_1(X) \rightarrow A^0(X, K_1)$
of the pullback of the structural morphism with the inclusions
\[K_1^{(0/1)}(X)=E_\infty^{0,-1}(X)\subset \cdots \subset E_3^{0,-1}(X)\subset E_2^{0,-1}(X)\]
given by the BGQ spectral sequence, is the identity. 
Therefore, for any $i\geq 2$, the differential starting from $E_i^{0,-1}(X)$ is zero, i.e
for any $i\geq 2$, one has 
\[E^{i, -i}_{i}(X)=\tau^{i/i+1}(X).\]
In particular, for $X$ a $\mathbf{SL_1}(A)$-torsor, with $A$ of prime degree $p$, one has
$E^{p+1, -p-1}_{p+1}(X)=0$, i.e the differential $\delta : A^1(X,K_2)\rightarrow \text{CH}^{p+1}(X)$
is surjective.
\end{rem}

\medskip

\subsection{On the group $A^1(X,K_2)$}
The proof in the next section will use 
the work of A.\,Merkurjev
on the \textit{Rost invariant} of simply connected algebraic groups (see \cite[Part II]{Coh}).
Let $X$ be a $\mathbf{SL_1}(A)$-torsor over $F$. The group $A^1(X_{F(X)},K_2)$ is infinite cyclic with generator $q$ and isomorphic to $A^1(\mathbf{SL_n},K_2)$ under restriction 
(where $n=\text{deg}(A)$). Furthermore, the restriction map 
$r:A^1(X,K_2)\rightarrow A^1(X_{F(X)},K_2)$ is injective with finite cokernel of same order as the element
$R_{\mathbf{SL_1}(A)}(X)$, where
 \[R_{\mathbf{SL_1}(A)}:H^1(F,\mathbf{SL_1}(A))\rightarrow H^3(F, \mathbb{Q}/\mathbb{Z}(2))\]
is the Rost invariant of $\mathbf{SL_1}(A)$ (see \cite[Theorem 9.10]{Coh}).
Moreover, the homomorphism $R_{\mathbf{SL_1}(A)}$ is of order $\text{exp}(A)$ 
by \cite[Theorem 11.5]{Coh}.

If $\text{char}(F)=l$ is prime then the modulo $l$ component 
$H^3(F, \mathbb{Z}/l\mathbb{Z}(2))$ of the Galois cohomology group 
$H^3(F, \mathbb{Q}/\mathbb{Z}(2))$ is the group $H_l^3(F)$ defined by
K. Kato in \cite{Ka1} by means of logarithmic differential forms.

\section{Proof of the result}

 In this section, we prove the result of this note.

\begin{thm}
\textit{Let $A$ be a central simple algebra of prime degree $p$ over a field $F$ and let $X$ be a $\mathbf{SL_1}(A)$-torsor.
Then}
\begin{enumerate}[(i)]
\item \textit{for any equidimensional $F$-variety $Y$, the change of field homomorphism}
\[\text{CH}(Y)\rightarrow \text{CH}(Y_{F(X)}),\]
\textit{where} $\text{CH}$ \textit{is the integral Chow group, is surjective in codimension $< p+1$.}

\item \textit{it is also surjective in codimension $p+1$ for a given $Y$ provided that the variety $X_{F(\zeta)}$ does not have any closed point of prime to $p$ degree for each generic point $\zeta \in Y$.}
\end{enumerate}
\end{thm}

\begin{proof}
We use notations and materials introduced in the previous section.
One can assume that $X$ does not have any rational point over $F$
(or equivalently $X$ does not have any closed point of prime to $p$ degree, by the result 
\cite[Theorem 3.3]{black} of J.\,Black), if else there
is nothing to prove. Note that in this situation, the central simple algebra $A$ is necessarily a division algebra.
We recall that conclusion (i) has already been proved (see Remark 2.1).
According to Proposition 1.2, it suffices to show that $\text{CH}^{p+1}(X_{F(\zeta)})=0$
for each generic point $\zeta \in Y$ to get conclusion (ii). Since
$X_{F(\zeta)}$ does not have any closed point of prime to $p$ degree, it is enough
to prove that $\text{CH}^{p+1}(X)=0$.

Assume on the contrary that $\text{CH}^{p+1}(X)\neq 0$. Then $\delta : A^1(X,K_2)\rightarrow \text{CH}^{p+1}(X)$ is nonzero (since $\delta$ is surjective by Remark 2.2), i.e $E_{p+1}^{1,-2}(X)$
is strictly included in $E_{p}^{1,-2}(X)=A^1(X,K_2)$.
We claim that this implies that, by denoting as $q_X$ the generator of $A^1(X,K_2)$, one has
$r(q_X)=q$. Indeed, otherwise one has $r(q_X)=p\cdot q$ by \S 2.3. Consecutively, by denoting
as $c$ the corestriction morphism $A^1(\mathbf{SL_p},K_2)\rightarrow A^1(X,K_2)$,
for any $i\geq 2$, one has $c(E_i^{1,-2}(\mathbf{SL_p}))=c(A^1(\mathbf{SL_p},K_2))=A^1(X,K_2)$ (where the first identity is due to $\text{CH}^i(\mathbf{SL_p})=0$ for any $i\geq 2$).
In particular, one has $E_{p}^{1,-2}(X)=c(E_{p+1}^{1,-2}(\mathbf{SL_p}))\subset E_{p+1}^{1,-2}(X)$,
which is a contradiction.

Therefore, we have shown that under the assumption $\text{CH}^{p+1}(X)\neq 0$, the generator
$q$ of $A^1(X_{F(X)},K_2)$ is rational.
Then it follows that the generator $g$ of $\text{CH}^{p+1}(X_{F(X)})$ is also rational.

However, since $A_{F(X)}$ is a still a division algebra, by \cite[Theorem 7.2 and Theorem 8.2]{528}, the cycle $g^{p-1}$ in $\text{CH}_0(\mathbf{SL_1}(A_{F(X)}))$ is nonzero and the latter group is cyclic
of order $p$ generated by the class of the identity of $\mathbf{SL_1}(A_{F(X)})$.
Thus, the degree of the rational cycle $g^{p-1}$ is prime to $p$.

It follows that $X$ has a closed point of prime to $p$ degree, which is a contradiction.

The Theorem is proved.
\end{proof}

\begin{rem}
The end of the above proof shows in particular that for a division algebra $A$ of prime degree $p$ over a
field $F$,
the kernel of the Rost invariant  $R_{\mathbf{SL_1}(A)}$ is trivial. This is already contained
in the result \cite[Theorem 12.2]{MS} of A.\,Merkurjev and A.\,Suslin under the assumption $\text{char}(F)\neq p$.
Indeed, let $\xi \in H^1(F,\mathbf{SL_1}(A))$ and let $X$ be the associated $\mathbf{SL_1}(A)$-torsor.
Assume that $R_{\mathbf{SL_1}(A)}(\xi)$ is trivial. It follows then by \S 2.3 that the generator of
$A^1(X_{F(X)},K_2)$ is rational. As we have seen in the above proof, this implies that $X$ 
has a rational point over $F$, i.e the cocycle $\xi$ is trivial.

Note also that for a division algebra $A$ of prime degree $p$ over a field $F$, the Rost invariant  $R_{\mathbf{SL_1}(A)}$ coincides, up to sign, with the normalized invariant given by the cup product 
$[A]\cup (c)\in H^3(F,\mathbb{Z}/p\mathbb{Z}(2))$ for any class $c\,\text{Nrd}(A^{\times})$, where $[A]$ is the class of the algebra $A$ 
in the Brauer group $\text{Br}(F)$, see \cite[\S 11]{Coh}.
\end{rem}

\section{Exceptional projective homogeneous varieties}

In this section, we describe how Theorem 1.1 implies a similar version of it for
projective homogeneous varieties under a group of type $F_4$ or $E_8$. 
Namely, we give an alternative proof of Theorem 4.1 below.
The following proof requires the characteristic of the base field to be different from
$p$, with $p=3$ when $G$ is of type $F_4$ and $p=5$ when $G$ is of type $E_8$, 
although the original result \cite[Theorem 1.1]{rcffephv} is valid for arbitrary characterisitic.

\medskip

Let $X$ be a nonsplit $\mathbf{SL_1}(A)$-torsor over a field $F$, with $A$ a division algebra of prime
degree $p$. There exists a smooth compactification $\Tilde{X}$ of $X$ such that the Chow motive
$\mathcal{M}(\Tilde{X},\mathbb{Z}/p\mathbb{Z})$ decomposes as a direct sum $\mathcal{R}_p\oplus N$,
where $\mathcal{R}_p$ is the indecomposable Rost motive associated with the symbol
$[A]\cup (c)\in H^3(F,\mathbb{Z}/p\mathbb{Z}(2))$, with $c\in F^{\times}\backslash \text{Nrd}(A^{\times})$ giving $X$, see \cite[Theorem 1.1]{528}.
Note that the projective variety $\Tilde{X}$ is a \textit{norm variety} of $s$.

\begin{thm}
\textit{Let $G$ be a linear algebraic group of type $F_4$ or $E_8$ over a field $F$
of characteristic different from $p$, 
with $p=3$ when $G$ is of type $F_4$ and $p=5$ when $G$ is of type $E_8$, 
and let $X'$ be a projective homogeneous $G$-variety.
For any equidimensional variety $Y$, the change of field homomorphism}
\[\text{Ch}(Y)\rightarrow \text{Ch}(Y_{F(X')}),\]
\textit{where} $\text{Ch}$ \textit{is the Chow group modulo $p$,
is surjective in codimension $<p+1$.}

\textit{It is also surjective in codimension $p+1$ for a given $Y$ provided that}
$1\notin \text{deg}\;\text{Ch}_0(X'_{F(\zeta)})$ \textit{for each generic point
$\zeta \in Y$.}
\end{thm}

\begin{proof}
Since the $F$-variety $X'$ is $A$-trivial in the sense of \cite[Definition 2.3]{OSNV},
one can assume that $G$ has no splitting field of degree coprime to $p$. 
Indeed, otherwise  $1\in \text{deg}\;\text{Ch}_0(X')$ by corestricition and 
this implies that $\text{Ch}(Y)\rightarrow \text{Ch}(Y_{F(X')})$ is an isomorphism
in any codimension by $A$-triviality, see \cite[Lemma 2.9]{OSNV}.

Let us now write $G={_{\xi}}G_0$ for 
a nontrivial cocycle $\xi \in H^1(F,G_0)$, with $G_0$ a split group of the same type as $G$.
Then the motive ${\mathcal{R}_p}(G)$ living on the Chow motive (with coefficients in $\mathbb{Z}/p\mathbb{Z}$) of $X'$ given in \cite[Theorem 5.17]{J-inv}
is the Rost motive of the symbol $R_{G_0,p}(\xi)=[A]\cup (c)\in H^3(F,\mathbb{Z}/p\mathbb{Z}(2))$,
where $R_{G_0,p}$ is the the modulo $p$ component of the Rost invariant $R_{G_0}$, $A$ is a division algebra of degree $p$ and $c\in F^{\times}\backslash  \text{Nrd}(A^{\times})$ 
 -- see \cite[\S 4]{F4} and \cite[\S 14]{Skip2} (here the assumption $\text{char}(F)\neq p$ is needed). 

Let us denote as $X$ the nonsplit $\mathbf{SL_1}(A)$-torsor over $F$ associated with $c$
and as $\Tilde{X}$ its smooth compactification.
We claim that $X'$ has a closed point of prime to $p$ degree over $F(\Tilde{X})$ and vice versa.

Indeed, since $\Tilde{X}$ is a norm variety for $[A]\cup (c)$, the motive ${\mathcal{R}_p}(G)$ decomposes
as a sum of Tate motives over $F(\Tilde{X})$. Therefore, the group $G_{F(\Tilde{X})}$ is split by an extension
of degree coprime to $p$ and 
it follows that $X'$ has a closed point of prime to $p$ degree over $F(\Tilde{X})$ 
(this is more generally true for any extension $L/F$ over which $\Tilde{X}$ has a closed point 
of prime to $p$ degree).
Moreover, the motive ${\mathcal{R}_p}(G)$ decomposes
as a sum of Tate motives over $F(X')$ because $G$ is split by $F(X')$. Consequently, $\Tilde{X}$ has a closed point of prime to $p$ degree over $F(X')$. 

It follows then (note that $\Tilde{X}$ is $A$-trivial by \cite[Example 5.7]{OSNV}) that the right and the bottom homomorphisms in the commutative square
\[\xymatrix{\text{Ch}(Y)\ar[r] \ar[d] & \text{Ch}(Y_{F(X')}) \ar[d] \\
 \text{Ch}(Y_{F(\Tilde{X})}) \ar[r] & \text{Ch}(Y_{F(\Tilde{X}\times X')}) }\]
are isomorphisms. Since $F(\Tilde{X})=F(X)$, Theorem 4.1 is now a direct consequence of Theorem 1.1.
\end{proof}

The following was pointed out to me by Philippe Gille.

\begin{rem}
Let $G_0$ a split group of type $E_8$ over a $5$-special field $F$ (i.e $F$ has no proper extension
of degree coprime to $5$) of characteristic $\neq 5$.
The above proof gives rise to a new argument for the triviality of the kernel of the Rost invariant modulo $5$
\[H^1(F,G_0)\rightarrow H^3(F,\mathbb{Z}/5\mathbb{Z}(2)).\]
This result is originally due to  Vladimir Chernousov (under the assumption $\text{char}(F)\neq 2$, $3$, $5$, see \cite[Theorem]{Ch1}).

Indeed, since $F$ is $5$-special, for any nontrivial cocycle $\xi\in H^1(F,G_0)$, the group
${_{\xi}}G_0$ has no splitting field of degree coprime to $5$. Then, as we have seen in the proof,
there is a division algebra $A$ of degree $5$ such that
$R_{G_0,5}(\xi)$ is equal to a symbol $[A]\cup (c)$ associated with a nonsplit 
$\mathbf{SL_1}(A)$-torsor $X$.
The injectivity of $R_{G_0,5}$ follows now from Remark 3.2.
\end{rem}

\bibliographystyle{acm} 
\bibliography{references}
\end{document}